\documentclass[12pt,twoside, draft]{article}
\usepackage{ifthen}
\usepackage{amsmath}
\usepackage{amsfonts}
\usepackage{latexsym}
\usepackage{amsthm}
\usepackage{amssymb}

\date{}

\begin{document}


\centerline{}

\centerline{}

\centerline {\Large{\bf Levy's phenomenon for analytic functions}}

\centerline{\Large{\bf  in the polydisc}}

\centerline{}

\centerline{\bf {A. O. Kuryliak}}

\centerline{}

\centerline{Department of Mechanics and Mathematics,}

\centerline{Ivan Franko National University of L'viv, Ukraine}

\centerline{kurylyak88@gmail.com}

\centerline{}

\centerline{\bf {O. B. Skaskiv}}

\centerline{}

\centerline{Department of Mechanics and Mathematics,}

\centerline{Ivan Franko National University of L'viv, Ukraine}

\centerline{matstud@franko.lviv.ua}

\centerline{}

\centerline{\bf {S. R. Skaskiv}}

\centerline{}

\centerline{Department of Mechanics and Mathematics,}

\centerline{Ivan Franko National University of L'viv, Ukraine}

\centerline{matstud@franko.lviv.ua}

\centerline{}

\newtheorem{Theorem}{\quad Theorem}[section]

\newtheorem{Definition}[Theorem]{\quad Definition}

\newtheorem{Corollary}[Theorem]{\quad Corollary}

\newtheorem{Lemma}[Theorem]{\quad Lemma}

\newtheorem{Example}[Theorem]{\quad Example}

\begin{abstract}
In this paper we prove some analogue of Wiman's  type inequality
for random analytic functions in the polydisc $\mathbb{D}^p=\{z\in\mathbb{C}^p\colon |z_j|<1, j\in\{1,\ldots,p\}\},\ p\in\mathbb{Z}_+$. The obtained
inequality is sharp.
\end{abstract}



\section{Introduction}

By $\mathcal{A}^1$ we denote the class of analytic functions in the disc $\mathbb{D}=\{z\colon |z|<1\},$
represented by power series
 \begin{equation}\label{1}
 f(z)=\sum^{+\infty}_{n=0}a_nz^n
 \end{equation}
 with radii of convergence $R(f)=1.$
 Let $M_f(r)=\max\{|f(z)|\colon|z|=r\}$ be maximum modulus and $\mu_f(r)=\max\{|a_n|r^n\colon
 n\geq  0\}$ maximal term of $f\in\mathcal{A}^1$, $ r\in[0;1).$

For analytic~function $f\in\mathcal{A}^1$
  and every $\delta>0$ there exists {a} set $E_f(\delta)\subset (0, 1)$ {of finite} logarithmic measure on $(0,1)$, i.e. $$ \int_{E_f(\delta)}\frac{dr}{1-r}<+\infty,$$  such that
  for all $r\in(0,1)\backslash E_f(\delta)$ the inequality
\begin{equation}\label{2}
M_f(r)\leq\frac{\mu_f(r)}{(1-r)^{1+\delta}}\ln^{1/2+\delta}\frac{\mu_f(r)}{1-r}
\end{equation}
holds. Similar inequality for analytic function
in the unit disc one can find in \cite{1, 2, 4, 5}.

Also (see \cite{2}) was proved the sharpness of inequality (\ref{2}). In particular,
\begin{equation}\label{star}
\varliminf_{r\to1-0} \frac{M_g(r)}{\frac{\mu_g(r)}{1-r}\ln^{1/2}\frac{\mu_g(r)}{1-r}}\geq C>0,\
g(z)=\sum_{n=1}^{+\infty}\exp\{n^{\varepsilon}\}z^n,\ \varepsilon\in(0,1).
\end{equation}

Using the Baire categories, in \cite{10, 11} was  described the ``quantity''
of those analytic functions $f\in\mathcal{A}^1$, for which inequality~(\ref{2}) can
be improved.

We start from two statement for random analytic functions in the unit disc.
Idea of proof of second statement will be used for proof of theorem 2.4.

 Let $\Omega=[0,1]$ and $P$ be the Lebesgue measure on $\mathbb{R}.$
   We consider the Steinhaus probability space $(\Omega,{\mathcal A},P),$ where
 $\mathcal{A}$ is the $\sigma$-algebra of Lebesgue measurable subsets of $\Omega$.
 Let ${X=(X_n(t))}$ be some sequence of random variables defined in this space.
{For an analytic function of the form
 $f(z)=\sum^{+\infty}_{n=0}a_nz^n$ by} $H(f,X)$ we denote the class of random analytic functions of the form
  \begin{equation}\label{4}
    f(z,t)=\sum_{n=0}^{+\infty} a_nX_n(t)z^n.
  \end{equation}

    In the sequel, the notion ``almost surely'' will be used in the
sense that the corres\-ponding property holds almost everywhere
with respect to Lebesgue measure  $P$ on $\Omega=[0,1]$. We say
that some relation holds almost surely in the class
  $H(f,X)$ if it holds for each analytic function $f(z,t)$ of the form (\ref{4}) almost
  surely in $t$.

Let $X=(X_n(t))$
 be multiplicative system (MS) uniformly bounded by the number 1.
 That is for all $n\in\mathbb{N}$ and $t\in[0,1]$ we have $|X_n(t)|\leq 1$
 for almost all $t\in[0;1]$
 and $$
 \forall (i_1,i_2,\ldots, i_n)\in\mathbb{N}^k,\ 1\leq i_1 <i_2<\cdots
 <i_k\ \colon\ {\bf M}(X_{i_1}X_{i_2}\cdots X_{i_k})=0,$$
 where ${\bf M}\xi$ is the expected value of a random variable $\xi$.

Similarly to \cite{4} one can prove such a statement.

\begin{Theorem}[\cite{10}]
  \sl
  Let $f(z,t)$ be random analytic function of form (\ref{4}), $X\in$ MS and $|X_n|\leq1$ for almost all
  $t\in[0;1]$.
  Then almost surely in $H(f,X)$ for any $\delta>0$
   there exists a set $E=E(f, t,\delta)\subset[0,1)$ of finite logarithmic measure on $[0;1)$
   $(\int_E\frac{dr}{1-r}<+\infty)$
  such that for all $r\in[0,1)\backslash E$ we have
  \begin{equation}\label{5}
    M_f(r,t)\leq\mu_f(r)\Biggl(\frac{1}{(1-r)^2}\cdot
    \ln\frac{\mu_f(r)}{1-r}\Biggl)^{1/4+\delta}.
  \end{equation}
 \end{Theorem}

Sharpness of inequality (\ref{5}) follows from such a statement.

\begin{Theorem}[\cite{10}]
  \sl
  Let $X\in$ be arbitrary sequence of random variables such that $|X_n|\geq1$ for almost all
  $t\in[0;1]$.
  Then there exist random analytic function $f(z,t)$ of form (\ref{4}) and a constants $C>0,\ 0<r_0<1$ such that almost surely in $H(f,X)$
  for $r\in(r_0,1)$ we have
  \begin{equation}\label{6}
    M_f(r,t)> C\mu_f(r)\Biggl(\frac{1}{(1-r)^2}\cdot
    \ln\frac{\mu_f(r)}{1-r}\Biggl)^{1/4},\ r\to1-0.
  \end{equation}
 \end{Theorem}

\section{Wiman's type inequality for analytic functions in the polydisc}
In \cite{8, 9, 11} one can find Wiman's type inequality for entire functions of
several complex variables.
Also in \cite{3} sharp Wiman's inequality was proved for analytic functions represented by the power series
\begin{equation*}
 f(z)=f(z_1,\ldots,z_p)=\sum_{\|n\|=0}^{+\infty}a_nz^n
 \end{equation*}
 with {the} domain
of convergence $\mathbb D^p=\{z\in\mathbb{C}^p\colon |z_j|<1,\ j\in\{1,\ldots,p\}\},$
 where $z^n=z_1^{n_1}\ldots z_p^{n_p},\ p\in\mathbb{N},\ p\geq2, \ n=(n_1,\ldots,n_p)\in \mathbb{Z}_+^p,\ \|n\|=\sum_{j=1}^pn_j.$

By $\mathcal{A}^p$ we denote the class of such analytic functions.

For $r=(r_1,r_2,\ldots, r_p)\in[0,1)^p$ and a function $f\in\mathcal{A}^p$ we denote
\begin{gather*}
\triangle_r=\{t\in[0,1)^p\colon t_j\geq r_j,\ j\in\{1,\ldots, p\}\},\\
M_f(r)=\max \{|f(z)|\colon  |z_j|\leq r_j,\ j\in\{1,\ldots, p\}\},\\
\mu_f(r)=\max\{|a_n|r^n\colon  n\in{\Bbb{Z}}_+^p\},\ \
\mathfrak{M}_f(r)=\sum_{\|n\|=0}^{+\infty}|a_n|r^n.
\end{gather*}

 We say that $E\subset[0,1)^p$ is a {\it set of asymptotically finite logarithmic measure on $[0,1)^p,$}
 if there exists $r_0\in[0,1)^p$ such that
 $$
 \nu_{\rm ln}(E\cap\triangle_{r_0}){:=}\idotsint\limits_{E\cap\triangle_{r_0}}\prod_{i=1}^p\frac{dr_i}{1-r_i}<+\infty,
 $$
i.e. the set $E\cap\triangle_{r_0}$  is a set of finite logarithmic measure on $[0,1)^p.$
We denote such class of the sets by $\Upsilon.$

For $f\in \mathcal{A}^p$ in \cite{3} was proved such a statement.
\begin{Theorem}\label{t1}
  \sl
  Let $f\in \mathcal{A}^p$. For every $\delta>0$
   there exists a set $E=E(f,\delta)\subset[0,1)^p, E\in\Upsilon$
  such that for all $r\in[0,1)^p\backslash E$ we have
  \begin{equation}\label{3}
    M_f(r)\leq\mathfrak{M}_f(r)\leq\mu_f(r)\Biggl(\prod_{j=1}^p\frac{1}{1-r_j}\cdot
    \ln^{p/2}\Biggl\{\mu_f(r)\prod_{j=1}^p\frac{1}{1-r_j}\Biggl\}\Biggl)^{1+\delta}.
  \end{equation}
 \end{Theorem}
 Also in \cite{3} was proved that exponent $1+\delta$
 in inequality (\ref{3}) cannot be replaced by a number
 smaller than 1. It follows from such a theorem.
  \begin{Theorem}
  \sl
  There exist a function $f\in \mathcal{A}^p,$ a constant $C>0$
  and a set $E\subset[0,1)^p, E\not\in\Upsilon$
  such that for all $r\in E$ the inequality
  \begin{equation*}
    M_f(r)\geq C\mu_f(r)\prod_{j=1}^p\frac{1}{1-r_j}\cdot
    \ln^{p/2}\Biggl\{\mu_f(r)\prod_{j=1}^p\frac{1}{1-r_j}\Biggl\}
  \end{equation*}
  holds.
 \end{Theorem}

 The aim of this paper is to prove the sharp Wiman's inequality
 for random analytic functions in the polydisc.
 We will prove, that almost surely the exponent $1+\delta$ in inequality (\ref{3}) one can replace
 by $\frac12+\delta,$ and this exponent cannot be placed by a number smaller than $\frac12.$

 Let $Z=(Z_{n}(t))$ be a complex sequence of random variables $Z_{n}(t)=X_{n}(t)+iY_{n}(t)$ such that both $X=(X_{n}(t))$ and $Y=(Y_{n}(t))$ are {real} MS and $K(f,Z)$ the class of random analytic functions of the form
 $$f(z,t)=\sum^{+\infty}_{\|n\|=0}a_{n}Z_{n}(t)z_1^{n_1}\ldots z_p^{n_p}.$$

For such a functions we prove following statement.

\begin{Theorem}
  \sl
  Let $f\in \mathcal{A}^p$, $Z$ be a MS uniformly bounded by the number 1, $\delta>0$.
  Then almost surely in $K(f,Z)$
   there exists a set $E=E(f,t,\delta), E\in\Upsilon$
  such that for all $r\in[0,1)^p\backslash E$ we have
  \begin{equation}\label{7}
    M_f(r,t)\leq \mu_f(r)\Biggl(\prod_{j=1}^p\frac{1}{\sqrt{1-r_j}}\cdot
    \ln^{p/4}\Biggl\{\mu_f(r)\prod_{j=1}^p\frac{1}{1-r_j}\Biggl\}\Biggl)^{1+\delta}.
  \end{equation}
 \end{Theorem}

We prove that no one of powers $1/2$ and $p/4$
in inequality (\ref{7})
we cannot replace by smaller number than $1/2$ and $p/4$ respectively.
It follows from such statement.

\begin{Theorem}
  \sl
  Let  $Z$ be a sequence of random variables such that $|Z_n|\geq1$ for almost all $t\in[0;1]$.
  Then
   there exist an analytic function $f\in \mathcal{A}^p$, a constant $C>0$ and a~set $E=E(f,t,\delta)\subset[0,1)^p,\ E\not\in\Upsilon$
  such that almost surely in $K(f,Z)$ for all $r\in E$  we get
  \begin{equation}\label{1s}
    M_f(r,t)\geq C\mu_f(r)\prod_{j=1}^p\frac{1}{\sqrt{1-r_j}}\cdot
    \ln^{p/4}\Biggl\{\mu_f(r)\prod_{j=1}^p\frac{1}{1-r_j}\Biggl\}.
  \end{equation}
 \end{Theorem}

\section{Proofs}
We give a proof of Theorem 1.2 for completeness.
\subsection{Proof of theorem 1.2.}

We consider
\begin{gather*}
g(z)=\sum_{n=1}^{+\infty}\exp\{\sqrt{n}\}z^n,\
f(z)=\sum_{n=1}^{+\infty}\exp\{\sqrt{n}/2\}z^n,\\
f(z,t)=\sum_{n=1}^{+\infty}X_n(t)\exp\{\sqrt{n}/2\}z^n.
\end{gather*}
Remark that for all $0<r<1$
$$
\mu_g(r^2)=\max\{e^{\sqrt{n}}r^{2n}\colon n\geq1\}=\max\{(e^{\sqrt{n}/2}r^{n})^{2}\colon n\geq1\}=(\mu_f(r))^2
$$
and using Parseval's equality we get
$$
M_g(r^2)\leq\sum_{n=1}^{+\infty}|X_n(t)|^2\exp\{\sqrt{n}\}r^{2n}=\frac{1}{2\pi}\int_{0}^{2\pi}|f(re^{i\theta},t)|^2d\theta\leq (M_f(r,t))^2.
$$
Therefore using (\ref{star}) we obtain
\begin{gather*}
(M_f(r,t))^2\geq M_g(r^2)\geq C\frac{\mu_g(r^2)}{1-r^2}\cdot
    \ln^{1/2}\frac{\mu_g(r^2)}{1-r^2}\geq \frac{C}2\frac{\mu_f^2(r)}{1-r}\cdot
    \ln^{1/2}\frac{\mu_f^2(r)}{1-r},\\
  M_f(r,t) \geq \sqrt{\frac{C}3}\frac{\mu_f(r)}{\sqrt{1-r}}\cdot
    \ln^{1/4}\frac{\mu_f(r)}{1-r}, \ r\to1-0.
\end{gather*}

\subsection{Proof of theorem 2.3.}

\begin{Lemma}[\cite{6}]{\sl Let $X=(X_{n}(t))$ be a MS
uniformly bounded by the number 1.
 Then for {each} $\beta>0$ there exists a constant
  $A_{\beta p}>0,$ which depends on $p$ and $\beta$ only such that for all
  $N\geq N_1(p)=\max\{p,4\pi\}$ and  $\{c_{n}\colon \|n\|\leq N\}\subset{\Bbb C}$
  we have
 \begin{gather}\label{6}
 P\Biggl\{t\colon\!\!\max\Biggl\{\Biggl|\sum_{\|n\|=0}^Nc_{n}X_{n}(t)e^{in_1\psi_1}\!\!\!\ldots
 e^{in_p\psi_p}\Biggl|\colon\!\psi\in[0,2\pi]^p\Biggl\}
 \geq A_{\beta p}S_N\ln^{\frac{1}{2}}N\Biggl\}\!\leq\!\frac{1}{N^\beta},
 \end{gather}
  where
  $S_N^2=\sum_{\|n\|=0}^N|c_{n}|^2.$}
\end{Lemma}

\begin{Lemma}[\cite{3}]
  \sl Let $\delta>0.$
  There exists a set $E\subset[0,1)^p, E\in\Upsilon$
  such that for all $r\in[0,1)^p\backslash E$ the inequality
  \begin{gather*}
  \frac{\partial}{\partial r_s}\ln\mathfrak{M}_f(r)\leq(\ln\mathfrak{M}_f(r))^{1+\delta}\cdot\frac1{1-r_s}\cdot\prod_{\begin{substack} {j=1 \\ j\neq s}\end{substack}}^p
  \Bigl(\frac1{1-r_j}\Bigl)^{\delta},\ s\in\{1,\ldots,p\},
  \end{gather*}
  holds.
\end{Lemma}
\begin{proof}[Proof of Theorem 2.3]
Without loss of generality we may suppose that
  $Z=X=(X_n(t))$ is a~MS (see \cite{7}).

For $k\in {\Bbb Z}_+$ and $l\in {\Bbb Z}$ such that $k>-l$ we denote
\begin{gather*}
G_{kl}=\Bigl\{r=(r_1,\ldots,r_p )\in [0;1)^p\colon\\
 k\leq \sum_{j=1}^p\ln\frac{1}{1-r_j}\leq k+1,\  l\leq \ln\mu_f(r)\leq l+1\Bigl\}, \ \
G_{kl}^+=\bigcup_{i=k}^{+\infty}\bigcup_{j=l}^{+\infty}G_{ij}.
\end{gather*}

Remark that the set
\begin{gather*}
E_0=\Bigl\{r\in[0;1)^p\colon \sum_{j=1}^p\ln \frac{1}{1-r_j}+\ln\mu_f(r)<1\Bigl\}=\\
=\Bigl\{r\in[0;1)^p\colon \mu_f(r)\prod_{j=1}^p\frac{1}{1-r_j}<e\Bigl\}\in\Upsilon,
\end{gather*}
because  there exists $r_0$ such that $E_0\cap[r_0;1)^p=\varnothing.$

By Lemma 3.2 there exists
a set $E_1\supset E_0, E_1\in\Upsilon$ such that for all
$r\in [0;1)^p\backslash E_1$ we have
\begin{gather*}
\sum_{\|n\|=0}^{+\infty}\|n\|\cdot|a_n|r^n\leq\mathfrak{M}_f(r)(\ln\mathfrak{M}_f(r))^{1+\delta}\cdot
\sum_{s=1}^p\Bigl(\frac1{1-r_s}\cdot\prod_{\begin{substack} {j=1 \\ j\neq s}\end{substack}}^p
  \Bigl(\frac1{1-r_j}\Bigl)^{\delta}\Bigl)\leq\\
\leq\mu_f(r)\Biggl(\prod_{j=1}^p\frac{1}{1-r_j}\cdot
    \ln^{p/2}\Biggl\{\mu_f(r)\prod_{j=1}^p\frac{1}{1-r_j}\Biggl\}\Biggl)^{1+\delta}\times\\
    \times \Bigl(\ln\mu_f(r)+(1+\delta)\Bigl(\sum_{j=1}^p\ln\frac1{1-r_j}+\frac{p}{2}\ln\Bigl(\ln\mu_f(r)+
    \sum_{j=1}^p\ln\frac1{1-r_j}\Bigl)\Bigl)\Bigl)^{1+\delta}\times\\
    \times \sum_{s=1}^p\Bigl(\frac1{1-r_s}\cdot\prod_{\begin{substack} {j=1 \\ j\neq s}\end{substack}}^p
  \Bigl(\frac1{1-r_j}\Bigl)^{\delta}\Bigl)\leq
  \\
  \leq\mu_f(r)\prod_{j=1}^p\frac{1}{(1-r_j)^{1+2\delta}}\cdot\sum_{j=1}^p\frac1{1-r_j}\cdot
  \ln^{p/2+1+p\delta}\Biggl\{\mu_f(r)\prod_{j=1}^p\frac{1}{1-r_j}\Biggl\}\leq\\
  \leq\mu_f(r)\prod_{j=1}^p\frac{1}{(1-r_j)^{2+2\delta}}\cdot
  \ln^{p/2+1+p\delta}\Biggl\{\mu_f(r)\prod_{j=1}^p\frac{1}{1-r_j}\Biggl\}.
\end{gather*}

Therefore
\begin{gather}\nonumber
\sum_{\|n\|\geq d}|a_n|r^n\leq \sum_{\|n\|\geq d}\frac{\|n\|}{d}|a_n|r^n
\leq\frac1{d}\sum_{\|n\|=0}^{+\infty}\|n\||a_n|r^n\leq\\
\label{8}
\leq\frac{1}{d}\mu_f(r)\prod_{j=1}^p\frac{1}{(1-r_j)^{2+2\delta}}\cdot
  \ln^{p/2+1+p\delta}\Biggl\{\mu_f(r)\prod_{j=1}^p\frac{1}{1-r_j}\Biggl\}\leq\mu_f(r),
  \end{gather}
where $$d=d(r)=\prod_{j=1}^p\frac{e^{2+3\delta}}{(1-r_j)^{2+3\delta}}\cdot
  \ln^{p/2+1+p\delta}\Biggl\{\mu_f(r)\prod_{j=1}^p\frac{1}{1-r_j}\Biggl\}.$$

   Let $G_{kl}^*=G_{kl}\setminus E_{2},$ $I=\{(i;j)\colon G_{ij}^*\neq\varnothing\},$
   $$
   E_{2}=E_1\cup\Biggl(\bigcup_{(i,j)\not\in I}G_{ij}\Biggl).
   $$

  Then
  $\#I=+\infty.$ For $(k,l)\in I$ we choose a sequence $r^{(k,l)}\in G_{kl}^*$
  such that
$
  \mu_f(r^{(k,l)})=\min\limits_{r\in G^*_{kl}}\mu_f(r).
$
 Then for all $r\in G_{kl}^*$ we get
  \begin{gather}\label{s1}
 \mu_f(r^{(k,l)})\leq\mu_f(r)\leq  e\mu_f(r^{(k,l)}),\\
 \label{s2}
\frac{1}{e} \prod_{j=1}^p\frac{1}{1-r_j^{(k,l)}}\leq\prod_{j=1}^p\frac{1}{1-r_j}\leq e\prod_{j=1}^p\frac{1}{1-r_j^{(k,l)}},\\
 \label{s3}
 \frac{1}{e^2}\mu_f(r^{(k,l)})\prod_{j=1}^p\frac{1}{1-r_j^{(k,l)}}\leq\mu_f(r)\prod_{j=1}^p\frac{1}{1-r_j}
 \leq  e^2 \mu_f(r^{(k,l)})\prod_{j=1}^p\frac{1}{1-r_j^{(k,l)}}
 \end{gather}
  and also
$$
\bigcup_{(k,l)\in I}G_{kl}^*=\bigcup_{(k,l)\in I}G_{kl}\setminus E_{1}=\bigcup_{k,l=1}^{+\infty}G_{kl}
\setminus E_{1}=[0;1)^p\setminus E_{1}.
$$
Denote $N_{kl}=[2d_1(r^{(k,l)})],$ where
 $$
 d_1(r)=\prod_{j=1}^p\frac{e^{2+3\delta}}{(1-r_j)^{2+3\delta}}\cdot
  \ln^{p/2+1+p\delta}\Biggl\{e^2\mu_f(r)\prod_{j=1}^p\frac{1}{1-r_j}\Biggl\}.
 $$
  For     $r\in G_{kl}^*$ we put
 $$
 W_{N_{kl}}(r,t)=\max\left\{\left|\sum_{\|n\|\leq
 N_{kl}}a_{n}r_1^{n_1}\ldots r_p^{n_p}e^{in_1\psi_1+\ldots+in_p\psi_p}X_{n}(t)\right|\colon\psi
 \in[0,2\pi]^p\right\}.
 $$
    For a Lebesgue measurable set $G\subset G_{kl}^*$ and for $(k,l)\in I$ we denote
 \begin{equation*}
  \nu_{kl}(G)=\frac{{\rm meas}_p(G)}{{\rm meas}_p(G_{kl}^*)},
 \end{equation*}
where meas$_p$ denotes the Lebesgue measure on $\mathbb{R}^p.$

 Remark that  $\nu_{kl}$ is a probability measure defined on the family of Lebesgue measurable subsets of
 $G_k^*$ (\cite{7}). Let
 $\Omega=\bigcup_{(k,l)\in I}G_{kl}^*$ and
 $$
 k_i, l_{i,j}\colon (k_i, l_{i,j})\in I,\ k_i<k_{i+1},\ l_{i,j}<l_{i,j+1},\
 \forall i,j\in\mathbb{Z}_+.
 $$
  For
 Lebesgue measurable subsets  $G$ of $\Omega$  we denote
 \begin{gather}\nonumber
 \nu(G)=2^{k_0}\sum_{i=0}^{+\infty}\Biggl(\frac{1}{2^{k_i}}\Big(1-\Big(\frac{1}{2}\Big)^{k_{i+1}
 -k_i}\Big)\times\\
 \label{9}
 \times\sum_{j=0}^{N_i}\frac{2^{l_{i,0}}}{2^{l_{i,j}}}\frac{\Big(1-\Big(\frac{1}{2}\Big)^{l_{i,j+1}
 -l_{i,j}}\Big)}{1-\big(\frac12\big)^{l_{i,N_i}}}\nu_{k_{i+1}l_{i+1,j+1}}(G\cap G^*_{k_{j+1}l_{i+1,j+1}})\Biggl),
  \end{gather}
  where $N_i=\max\{j\colon (k_i,l_{ij})\in I\}.$
  Remark that $\nu_{k_{j+1}l_{j+1}}(G^*_{k_{j+1}l_{j+1}})=\nu(\Omega)=1.$

   Thus $\nu$ is a probability measure, which is defined on measurable
   subsets of
     $\Omega.$ On $[0,1]\times\Omega$
 we define the probability measure $P_0=P\otimes\nu,$
 which is a direct product of the probability measures $P$ and $\nu.$
 Now for $(k;l)\in I$ we define
 \begin{gather*}
 F_{kl}=\{(t,r)\in
 [0,1]\times\Omega\colon
 W_{N_{kl}}(r,t)>A_{p}S_{N_{kl}}(r)\ln^{1/2}N_{kl}\},\\
 F_{kl}(r)=\{t\in[0,1]\colon
 W_{N_{kl}}(r,t)>A_{p}S_{N_{kl}}(r)\ln^{1/2}N_{kl}\},
 \end{gather*}
  where $S^2_{N_{kl}}(r)=\sum_{\|n\|=0}^{N_{kl}}|a_{n}|^2r^{2n}$
  and $A_p$ is the constant from Lemma 3.2 with $\beta=1.$ Using Fubini's theorem and Lemma 1 with
 $c_n=a_nr^n$ and $\beta=1,$ we get for
 $(k,l)\in I$
 $$
 P_0(F_{kl})=\int\limits_{\Omega}\Bigg(\int\limits_{F_{kl}(r)}dP\Bigg)d\nu=\int\limits_{\Omega}P(F_{kl}(r))d\nu
 \leq\frac{1}{N_{kl}}\nu(\Omega)=\frac{1}{N_{kl}}.
 $$

 Note that
 $$
 N_{kl}>\prod_{j=1}^p\frac{1}{(1-r_j^{(k,l)})^2}\ln^{p/2+1+p\delta}\Bigl\{\mu_f(r^{(k,l)})\prod_{j=1}^p\frac{1}{1-r_j^{(k,l)}}\Bigl\}\geq e^k(l+k)^{2+p\delta}.
 $$
 Therefore
$$
\sum_{(k,l)\in I}P_0(F_{kl})\leq\sum_{k=1}^{+\infty}\sum_{l=-k+1}^{+\infty}\frac{1}{e^k(l+k)^{2+p\delta}}<+\infty.
$$

 By Borel-Cantelli's lemma the infinite quantity
of the events $\{F_{kl}\colon (k,l)\in I\}$ may occur with probability zero. So,
$$P_0(F)=1, \quad
F=\bigcup_{s=1}^{+\infty}\bigcup_{m=1}^{+\infty}\bigcap_{\begin{substack} {k\geq  s,\  l\geq m \\ (k,l)\in
I}\end{substack}}\overline{F_{kl}}\subset [0,1]\times \Omega.$$

Then for any point
$(t,r)\in F$ there exist $k_0=k_0(t,r)$ and $l_0=l_0(t,r)$ such that for all
$k\geq k_0,$ $l\geq l_0,$ $(k,l)\in I$ we have
\begin{equation}\nonumber
 W_{N_{kl}}(r,t)\leq A_{p}S_{N_{kl}}(r)\ln^{1/2}N_{kl}.
 \end{equation}

%

So, $\nu(F^{\wedge}(t))=1$ (see \cite{7}).


  For any $t\in F_1$(\cite{7}) and $(k,l)\in I$ we choose a point $r_0^{(k,l)}(t)\in
G_{kl}^*$ such that
$$
W_{N_{kl}}(r_0^{(k,l)}(t),t)\geq\frac{3}{4}M_{kl}(t),\
M_{kl}(t)\stackrel{{\rm def}}{=} \sup\{W_{N_{kl}}(r,t)\colon r\in G_{kl}^*\}.
$$

Then from $\nu_{kl}(F^{\wedge}(t)\cap G_{kl}^*)=1$ for all $(k,l) \in I$ it follows that
there exists a point  $r^{(k,l)}(t)\in G_{kl}^*\cap
F^{\wedge}(t)$ such that
$$
|W_{N_{kl}}(r_0^{(k,l)}(t),t)-W_{N_{kl}}(r^{(k,l)}(t),t)|<\frac{1}{4}M_{kl}(t)
$$
  or
  $$
  \frac{3}{4}M_{kl}(t)\leq W_{N_{kl}}(r_0^{(k,l)}(t),t)\leq
W_{N_{kl}}(r^{(k,l)}(t),t)+\frac{1}{4}M_{kl}(t).
$$
 Since
$(t,r^{(k,l)}(t))\in F,$  from inequality (\ref{9}) we obtain
\begin{equation} \nonumber
\frac{1}{2}M_{kl}(t)\leq
W_{N_{kl}}(r^{(k,l)}(t),t)\leq  A_p
 S_{N_{kl}}(r^{(k,l)}(t))\ln^{1/2}N_{kl}.
\end{equation}
 Now for
$r^{(k,l)}=r^{(k,l)}(t)$ we get
\begin{gather*}
S^2_{N_{kl}}(r^{(k,l)})\leq\mu_f(r^{(k,l)})\mathfrak{M} _f(r^{(k,l)})\leq\\
\leq
\mu_f^2(r^{(k,l)})\Biggl(\prod_{j=1}^p\frac{1}{1-r_j^{(k,l)}}
    \ln^{p/2}\Biggl\{\mu_f(r^{(k,l)})\prod_{j=1}^p\frac{1}{1-r_j^{(k,l)}}\Biggl\}\Biggl)^{\!\!1+\delta}\!\!\!.
\end{gather*}
So, for  $t\in  F_1$ and all  $k\geq k_0(t),$ $l\geq l_0(t),$ we obtain
\begin{equation}\label{13}
S_N(r^{(k,l)})\leq\mu_f(r^{(k,l)})\Biggl(\prod_{j=1}^p\frac{1}{1-r_j^{(k,l)}}
    \ln^{p/2}\Biggl\{\mu_f(r^{(k,l)})\prod_{j=1}^p\frac{1}{1-r_j^{(k,l)}}\Biggl\}\Biggl)^{1/2+\delta/2}.
 \end{equation}

  It follows from (\ref{s1})--(\ref{s3}) that $d_1(r^{(k,l)})\geq d(r)$ for $r\in G_{kl}^*.$ Then for
  $t\in F_1,$ $r\in
F^{\wedge}(t)\cap  G_{kl}^*,$ $(k,l)\in I,\ k\geq k_0(t), l\geq l_0(t)$ we get
$$
M_f(r,t)\leq
 \sum_{\|n\|\geq
2d_1(r^{(k,l)})}|a_{n}|r^n+W_{N_{kl}}(r,t)\leq
\sum_{\|n\|\geq2d(r)}|a_{n}|r^n+M_{kl}(t).
$$

  Finally for $t\in F_1,$  $r\in
F^{\wedge}(t)\cap G_{kl}^*,$ $l\geq l_0(t)$ and $k\geq k_0(t)$  we obtain
\begin{gather*}
M_f(r^{(k,l)},t)\leq \mu_f(r^{(k,l)})+2A_{p}S_{N_{kl}}(r^{(k,l)})\ln^{1/2}N_{kl}\leq\mu_f(r^{(k,l)})+\\
+
2A_{p}\mu_f(r^{(k,l)})\Biggl(\prod_{j=1}^p\frac{1}{1-r_j^{(k,l)}}
    \ln^{p/2}\Biggl\{\mu_f(r^{(k,l)})\prod_{j=1}^p\frac{1}{1-r_j^{(k,l)}}\Biggl\}\Biggl)^{1/2+\delta/2}\times\\
    \times\ln\Biggl(6\prod_{j=1}^p\frac{1}{(1-r_j^{(k,l)})^{2+3\delta}}\cdot
  \ln^{p/2+1+p\delta}\Biggl\{e^2\mu_f(r^{(k,l)})\prod_{j=1}^p\frac{1}{1-r_j^{(k,l)}}\Biggl\}\Biggl).
\end{gather*}
So,
we get for $t\in F_1,$ $r\in F^{\wedge}(t)\cap G_{kl}^*,$ $k\geq
 k_0(t)$ and
$l\geq
 l_0(t)$
  \begin{equation}\label{14}
   M_f(r,t)\leq \mu_f(r)\Biggl(\prod_{j=1}^p\frac{1}{(1-r_j)^{1/2}}\cdot
    \ln^{p/4}\Biggl\{\mu_f(r)\prod_{j=1}^p\frac{1}{1-r_j}\Biggl\}\Biggl)^{1+\delta}.
\end{equation}

Therefore inequality (\ref{14}) holds almost surely
($t\in F_1,$ $P(F_1)=1$) for all
\begin{gather*}
r\in \Bigl(\bigcup_{(k,l)\in I}(G_{kl}^*\cap F^{\wedge}(t))\cap
G_{kl}^+\Bigl)\setminus E^*=\\
= ([0;1)^p\cap
G^+_{kl})\setminus(E^*\cup G^*\cup E_{1})
=[0;1)^p\setminus E_{2},
\end{gather*}
where
$$
G^+_{kl}=\bigcup_{i=k}^{+\infty}\bigcup_{j=l}^{+\infty}G_{kl},\
E_{2}=E_{1}\cup G^*\cup E^*,\ G^*=\bigcup_{(k,l)\in I}(G_{kl}^*\setminus
F^{\wedge}(t)).
$$

  It remains to remark that $\nu(G^*)$ satisfies
$\nu(G^*)=\sum_{(k,l)\in I}(\nu_{kl}(G_{kl}^*)-\nu_{kl}(F^{\wedge}(t)))$ $=0.$ Then for all $(k,l)\in I$ we obtain
\begin{gather*}
\nu_{kl}(G_{kl}^*\setminus
F^{\wedge}(t))=\frac{{\rm meas}_p(G_{kl}^*\setminus
F^{\wedge}(t))}{{\rm meas}_p(G_{kl}^*)}=0,\\
{\rm
meas}_p(G_{kl}^*\setminus F^{\wedge}(t))=\idotsint\limits_{G_{kl}^*\setminus
F^{\wedge}(t)}\frac{dr_1\ldots dr_p}{(1-r_1)\ldots (1-r_p)}=0.
\end{gather*}
\end{proof}

\subsection{Proof of theorem 2.4.}
Consider the function
$$
f(z)=\prod_{j=1}^pf_0(z_j),\ z=(z_1,\ldots, z_p)\in\mathbb{D}^p,\
f_0(\tau)=\sum_{k=1}^{+\infty}e^{\sqrt{k}/2}\tau^{k}, \ \tau\in\mathbb{D}.
$$

The function $g(t)=\ln\frac{\mu_{f_0}(t)}{1-t}$ is positive continuous
increasing on $(1/2;1)$, $\lim\limits_{t\to1-0}g(t)=+\infty.$
Therefore, there exists the inverse function $g^{-1}\colon \mathbb{R}_+\to(1/2;1).$

For the function $f(z)$ and $r\in[0;1)^p$ we have
$$
M_f(r)=\prod_{j=1}^pM_{f_0}(r_j),\ \mu_f(r)=\prod_{j=1}^p\mu_{f_0}(r_j).
$$

It follows from Theorem 1.2, that there exist $t'\in (0;1)$ and a constant $C_1>0$ such that for $t\in(t',1)$ we get
\begin{equation}
M_{f_0}(t)\geq C_1\frac{\mu_{f_0}(t)}{\sqrt{1-t}}\ln^{1/4}\frac{\mu_{f_0}(t)}{1-t}.
\end{equation}

Let us prove inequality
\begin{equation}\label{2s}
g^{-1}(3g(t))-g^{-1}\Bigl(\frac{g(t)}{3}\Bigl)>1-g^{-1}(3g(t)),\ t\to1-0.
\end{equation}

For fixed $t\in(0;1)$ we consider the function $l(x)=\frac12\sqrt{x}-x\ln\frac{1}{t}.$
$x_{\max}=\frac{1}{16\ln^2\frac{1}{t}}$ is unique maximum point of the function $l(x)$.
Thus
\begin{gather*}
\max\{l(x)\colon x>0\}=l_{\max}=\frac{1}{16\ln\frac{1}{t}},
\\
g(t)=\ln\frac{\mu_{f_0}(t)}{1-t}\sim\ln\mu_{f_0}(t)\sim\frac{1}{16\ln\frac{1}{t}}\sim\frac{1}{16(1-t)},\ t\to1-0.
\end{gather*}

Then $g(t)<3g(2t-1),\ t\to1-0.$
Therefore,
$$
g(2t-1)>\frac{g(t)}{3},\
2t-1>g^{-1}\Bigl(\frac{g(t)}{3}\Bigl),\
t-g^{-1}\Bigl(\frac{g(t)}{3}\Bigl)>1-t.
$$
Using
$g^{-1}(3g(t))>g^{-1}(g(t))=t$ we obtain as $t\to1-0$
$$
g^{-1}(3g(t))-g^{-1}\Bigl(\frac{g(t)}{3}\Bigl)>t-g^{-1}\Bigl(\frac{g(t)}{3}\Bigl)>1-t>1-g^{-1}(3g(t)).
$$
So, inequality (\ref{2s}) is proved.

There exist a constant $C_1\in(0,1)$ and
$r^*\in(r',1)$ such that for all $z\in\{z\colon t^*<|z_k|<1,$ $k\in\{1,\ldots,p\}\}$ we obtain
\begin{equation}\label{8}
  M_{f_0}(r_k)\geq C_1\frac{\mu_{f_0}(r_k)}{\sqrt{1-r_k}}
  \ln^{1/4}\frac{\mu_{f_0}(r_k)}{1-r_k}
  \ \mbox{ and }\ g^{-1}\Bigl(\frac{g(t^*)}{3}\Bigl)>r^0.
\end{equation}

Then for all $z\in\{z\colon r^*<|z_j|<1,\ j\in\{1,\ldots,p\}\}$ we have
\begin{gather}
\nonumber
\prod_{i=1}^pM_{f_0}(r_i)\geq\prod_{i=1}^p\Biggl(C_1\frac{\mu_{f_0}(r_i)}{\sqrt{1-r_i}}
  \ln^{1/4}\frac{\mu_{f_0}(r_i)}{1-r_i}\Biggl),\\
  \label{3s}
  M_f(r)\geq C_1^p\mu_f(r)\prod_{i=1}^p\frac{1}{\sqrt{1-r_i}}
  \Biggl(\prod_{i=1}^p\ln\frac{\mu_{f_0}(r_i)}{1-r_i}\Biggl)^{1/4}.
\end{gather}

For $r_1\in(t^*,1)$ we define $x$ and $y$ such that
$$
x=x(r_1)=g^{-1}\Bigl(\frac{g(r_1)}{3}\Bigl),\  y=y(r_1)=g^{-1}(3g(r_1)).
$$

Let
$
E^*=\{r\in[0,1)^p\colon r_1\in(t^*,1),\ r_i\in(x, y),\ i\in\{2,\ldots,p\}\}.
$

Fix $r_1\in(r^*,1).$ Then $x$ and $y$ are also fixed and
 $$
 g(x)=g(r_1)/3,\ g(y)=3g(r_1),\ g(y)=9g(x),\ (r_2,\ldots,r_p)\in(x,y)^{p-1}.
 $$
  Therefore, using that $r_1>x$ we get for all $r\in E^*$
\begin{gather*}
  \prod_{i=1}^pg(r_i)\geq g^{p}(x)=\frac{g^{p}(y)}{9^p}=
  \frac{1}{(9p)^p}(\underbrace{g(y)+\ldots+g(y)}_{p})^p\geq\\
  \geq\frac{1}{(9p)^p}(g(r_1)+\ldots+g(r_p))^p=\frac{1}{(9p)^p}\Biggl(\sum_{i=1}^pg(r_i)\Biggl)^p.
\end{gather*}

Therefore we get for all $r\in E^*$
\begin{gather*}
M_f(r)\geq C_1^p\mu_f(r)\prod_{i=1}^p\frac{1}{\sqrt{1-r_i}}\cdot\frac{1}{(9p)^p}
\Biggl(\sum_{i=1}^p\ln\frac{\mu_{f_0}(r_i)}{1-r_i}\Biggl)^{p/4}=\\
=C_2\mu_f(r)\prod_{i=1}^p\frac{1}{\sqrt{1-r_i}}\ln^{p/4}\Biggl(\mu_f(r)\prod_{i=1}^p\frac{1}{1-r_i}\Biggl).
\end{gather*}

It remains to prove, that set $E^*$ is a set of infinite asymptotically logarithmic measure.
Since $g^{-1}\Bigl(\frac{g(r^*)}{3}\Bigl)>r^0$ then $E^*\cap\triangle_{r^0}=E^*.$ Finally
\begin{gather*}
\nu_{\rm ln}(E^*\cap\triangle_{r^0})=\nu_{\rm ln}(E^*)=
\idotsint\limits_{E^*}\prod_{i=1}^p\frac{dr_i}{1-r_i}
=
\int\limits_{t^*}^1\underbrace{\int\limits_{x}^{y}\ldots\int\limits_{x}^{y}}_{p-1}\prod_{i=1}^p\frac{dr_i}{1-r_i}=\\
=\int\limits_{t^*}^1\Biggl(\int\limits_{x}^{y}\frac{dr_2}{1-r_2}\Biggl)^{p-1}\frac{dr_1}{1-r_1}
=\int\limits_{t^*}^1\Biggl(\ln\frac{1}{1-y}-\ln\frac{1}{1-x}\Biggl)^{p-1}\frac{dr_1}{1-r_1}=
\\
=\int\limits_{t^*}^1\Biggl(\ln\frac{1}{1-g^{-1}(3g(r_1))}-\ln\frac{1}{1-g^{-1}(\frac{g(r_1)}{3})}\Biggl)^{p-1}\frac{dr_1}{1-r_1}=
\\
=\int\limits_{t^*}^1\ln^{p-1}\frac{1-g^{-1}(\frac{g(r_1)}{3})}{1-g^{-1}(3g(r_1))}\frac{dr_1}{1-r_1}=
\\
=\int\limits_{t^*}^1\ln^{p-1}\Bigl(1+\frac{g^{-1}(3g(r_1))-g^{-1}(\frac{g(r_1)}{3})}{1-g^{-1}(3g(r_1))}\Bigl)\frac{dr_1}{1-r_1}
>\int\limits_{t^*}^1\ln^{p-1}2\cdot\frac{dr_1}{1-r_1}=+\infty.
\end{gather*}


\end{document}